\documentclass{amsart}[12pt]
\usepackage{amssymb,amsmath}
\usepackage{enumerate}
\usepackage[english]{babel}
\usepackage{color}

\newcommand {\ep} {\epsilon}

\newcommand {\ii} {\infty}
\newcommand {\dt} {\delta}
\newcommand {\al} {\alpha}
\newcommand {\bt} {\beta}
\newcommand {\lb} {\lambda}

\newcommand {\su} {\subset}
\newcommand {\wt} {\widetilde}

\newcommand {\mc} {\mathcal}
\newcommand{\gm} {\gamma}

\newcommand{\PM} {\mc P(\mc M)}

\newtheorem{teo}{Theorem}[section]
\newtheorem{pro}{Proposition}[section]
\newtheorem{cor}{Corollary}[section]

\newtheorem{lm}{Lemma}[section]

\theoremstyle{definition}
\newtheorem{rem}{Remark}[section]
\newtheorem{df}{Definition}[section]

\title{Individual Ergodic Theorems in Noncommutative Orlicz Spaces}
\begin{document}
\date{January 31, 2016}

\begin{abstract}
For a noncommutative Orlicz space associated with a semifinite von Neumann algebra, a faithful normal semifinite trace and an Orlicz function satisfying
$(\dt_2,\Delta_2)-$condition, an individual ergodic theorem is proved.
\end{abstract}

\author{VLADIMIR CHILIN, \ SEMYON LITVINOV}
\address{The National University of Uzbekistan, Tashkent, Uzbekistan}
\email{vladimirchil@gmail.com; chilin@ucd.uz}
\address{Pennsylvania State University \\ 76 University Drive \\ Hazleton, PA 18202, USA}
\email{snl2@psu.edu}

\maketitle
\section{Introduction}
Development of the theory of noncommutative integration with respect to a faithful normal semifinite trace $\tau$ defined on a semifinite von Neumann algebra $\mc M$, has given rise to a systematic study of various classes of  noncommutative  rearrangement invariant  Banach spaces. The noncommutative $L^p-$spaces $L^p(\mc M,\tau)$ \cite{se, ye1, px} and, more generally,  noncommutative Orlicz spaces $L^\Phi(\mc M,\tau)$ \cite{m,m1,ku} are important examples of such spaces.

Since every $L^\Phi(\mc M,\tau)$ is an exact interpolation space for the Banach couple $(L^1(\mc M,\tau),\ \mc M)$,
for any linear operator $T: L^1(\mc M,\tau)+\mc M\to L^1(\mc M,\tau)+\mc M$ such that
$$
\| T(x)\|_{\ii}\leq \| x\|_{\ii} \ \ \forall \ x\in \mc M \text{ \ and \ } \| T(x)\|_1\leq \| x\|_1 \ \ \forall \ x\in L^1(\mc M,\tau)
$$
(such operators are called Dunford-Schwartz operators), we have
$$
T(L^\Phi) \su L^\Phi  \text{ \ and \ } \|T\|_{ L^\Phi \to L^\Phi} \leq 1.
$$
Thus, it is natural to study noncommutative Dunford-Shwartz ergodic theorem in $L^\Phi(\mc M,\tau)$.
The first result in this direction was obtained in \cite{ye} for the space $L^1(\mc M,\tau)$ (as it is noticed in \cite[Proposition 1.1]{cl},
the class of operators $\al$ that was employed in \cite{ye} coincides with the class of positive Dunford-Schwartz operators).
In \cite{jx}, the result of \cite{ye} was extended to the noncommutative $L^p-$spaces with $1<p< \infty$.
For general noncommutative fully symmetric  spaces with non trivial Boyd indexes, an individual ergodic theorem was established in \cite{cl}.

Note that the class of Orlicz spaces $L^\Phi(\mc M,\tau)$ is significantly wider than the class of spaces $L^p(\mc M,\tau)$. Besides,
there are Orlicz spaces $L^\Phi(\mc M,\tau)$, with the Orlicz function satisfying the so-called $(\dt_2,\Delta_2)-$condition, which
have trivial  Boyd index $p_{L^\Phi} =1$ (see Remark \ref{r3} below). Therefore an individual ergodic theorem for positive Dunford-Schwartz operators in
Orlicz spaces does not follow from the results mentioned above.

The aim of this article is to establish an individual ergodic theorem for a positive Dunford-Schwartz operator in
a noncommutative Orlicz space $L^\Phi(\mc M, \tau)$
associated with an Orlicz function $\Phi$ satisfying $(\dt_2,\Delta_2)-$condition. Our argument is essentially based on the notion of uniform equicontinuity in measure at zero of a sequence of linear maps from a normed space into the space of measurable operators
affiliated with $(\mc M,\tau)$. This notion was introduced in \cite{cl1} and then applied in \cite{li} to provide a simplified proof
of noncommutative individual ergodic theorem for positive Dunford-Schwartz operators in $L^p(\mc M,\tau)$, $1<p<\ii$.

\section{Preliminaries}

Assume that $\mc M$ is a semifinite von Neumann algebra with a faithful normal semifinite trace $\tau$,
and let $\mc P(\mc M)$ be the complete lattice of projections in $\mc M$.
If $\mathbf 1$ is the multiplicative identity of $\mc M$ and $e\in \mc P(\mc M)$,
we denote $e^{\perp}=\mathbf 1-e$.
Let $L^0=L^0(\mc M,\tau)$ be the $*$-algebra of $\tau$-measurable operators. Recall that $L^0$ is a
metrizable topological $*$-algebra with respect to the {\it measure topology} that can be equivalently
(see \cite[Theorem 2.2]{cls}) defined by either of the families
$$
V(\ep,\dt)=\{ x \in L^0: \|xe\|_{\ii}\leq \dt \text { \ for
some } e\in \PM \text {\ with } \tau (e^{\perp})\leq \ep \}
$$
or
$$
   W(\ep,\dt)=\{ x\in L^0 : \|exe\|_{\ii}\leq
\dt \text { \ for some } e\in \PM \text {\ with } \tau
(e^{\perp})\leq \ep \},
$$
$\ep>0, \ \dt>0$, of  neighborhoods of zero \cite{ne}.

For a positive  operator
$x=\int_{0}^{\infty}\lambda de_{\lambda} \in L^0$  one can define
$$
\tau(x)=\sup_{n}\tau \left (\int_{0}^{n}\lambda de_{\lambda}\right
)=\int_{0}^{\infty} \lb d\tau(e_{\lambda}).
$$
If $1\leq p< \infty$, then the noncommutative $L^p-$space
associated with $(\mathcal M, \tau)$ is defined as
$$
L^p=(L^p(\mc M,\tau), \| \cdot \|_p) =\{ x\in L^0: \| x\|_p=(\tau
(| x|^p))^{1/p}<\ii\} ,
$$
where $|x|=(x^*x)^{1/2}$ is the absolute value of $x$; naturally, $L^{\ii}=\mc M$.

For detailed accounts on noncommutative $L^p$-spaces, see \cite{px,ye1}.

Given $x\in L^0$, let $\{ e_{\lb}\}_{\lb\ge 0}$ be the spectral family of projections of $| x|$.
If $t>0$, the {\it $t$-th generalized singular number} of $x$ \cite{fk} is defined as
$$
\mu_t(x)=\inf\{\lb>0: \tau(e_{\lb}^{\perp})\leq t\}.
$$

A Banach space $(E, \| \cdot \|_E) \su L^0$ is called {\it fully symmetric} if the conditions
$$
x\in E, \ y\in L^0, \ \int \limits_0^s\mu_t(y)dt\leq  \int \limits_0^s\mu_t(x)dt \ \ \forall \ s>0
$$
imply that $y\in E$ and $\| y\|_E\leq \| x\|_E$.

If $L \subset L^0$, the set of all positive operators in $L$ will be denoted by $L_+$.

A  fully symmetric space $(E, \| \cdot \|_E)$ is said to possess
{\it Fatou property} if the conditions
$$
x_{\al}\in E_+, \ \ x_{\al}\leq x_{\bt} \text{ \ for } \al \leq \bt, \text{\ and\ }\sup_{\al} \| x_{\al}\|_E<\ii
$$
imply that there exists $x=\sup \limits_{\al}x_{\al}\in E$ and $\| x\|_E=\sup \limits_{\al} \| x_{\al}\|_E$.

Let $m$ be Lebesgue measure on the interval $(0,\ii)$, and let $L^0(0,\ii)$ be the linear space of all
(equivalence classes of) almost everywhere finite
complex-valued $m-$measurable functions on $(0,\ii)$.
We identify $L^{\ii}(0,\ii)$ with the commutative von Neumann algebra acting
on the Hilbert space $L^2(0,\ii)$ via
multiplication by the elements from $L^{\ii}(0,\ii)$ with the trace given by the integration
with respect to Lebesgue measure.
A fully symmetric space $E\su L^0(\mc M,\tau)$, where $\mc M=L^{\ii}(0,\ii)$ and $\tau$ is given by the Lebesgue
integral, is called {\it fully symmetric function space} on $(0,\ii)$.

Let $E=(E(0,\ii), \| \cdot \|_E)$ be a fully symmetric function space.
For each $s>0$ let $D_s: E(0,\infty) \to E(0,\infty)$ be the bounded linear operator given by
$$
D_s(f)(t) = f(t/s), \ t > 0.
$$
The {\it Boyd indices} $p_E$ and $q_E$ are defined as
$$
p_E=\lim\limits_{s\to\infty}\frac{\log s}{\log \|D_s\|_E}, \ \ q_E=\lim\limits_{s \to +0}\frac{\log s}{\log \|D_{s}\|_E}.
$$
It is known that $1\leq p_E\leq q_E\leq \ii$ \cite[II, Ch.2, Proposition 2.b.2]{lt}.
A fully symmetric function space is said to have {\it non-trivial Boyd indices} if $1<p_E$ and $q_E<\ii$.
For example, the spaces $L^p(0,\ii)$, $1< p<\ii$, have non-trivial Boyd indices:
$$
p_{L_p(0,\infty)} = q_{L_p(0,\infty)} = p
$$
\cite[II, Ch.2, 2.b.1]{lt}.

If $E$ is a fully symmetric function space  on $(0,\ii)$, define
$$
E(\mc M)=E(\mc M, \tau)=\{ x\in L^0(\mc M,\tau): \ \mu_t(x)\in E\}
$$
and set
$$
\| x\|_{E(\mc M)}=\| \mu_t(x)\|_E,  \ x\in E(\mc M).
$$
It is shown in \cite{ddp}  that $(E(\mc M), \| \cdot \|_{E(\mc M)})$ is a fully symmetric space.

If $1\leq p<\ii$ and $E=L^p(0,\ii)$, the space $(E(\mc M), \| \cdot \|_{E(\mc M)})$ coincides with the noncommutative
$L^p$-space $(L^p(\mc M, \tau), \| \cdot \|_p)$ because
$$
\| x\|_p=\left (\int \limits_0^{\ii}\mu_t^p(x)dt\right )^{1/p}=\| x\|_{L^p(\mc M, \tau)}.
$$
\cite[Proposition 2.4]{ye1}.

Since for a fully symmetric function space $E$ on $(0,\ii)$,
$$
L^1(0,\ii) \cap L^{\infty}(0,\ii) \su E \su L^1(0,\ii) + L^{\infty}(0,\ii)
$$
with continuous embeddings \cite[Ch.II, \S 4, Theorem 4.1]{kps}, we also have
$$
L^1(\mc M, \tau) \cap \mc M \subset E(\mc M, \tau) \subset L^1(\mc M, \tau) + \mc M,
$$
with continuous embeddings.

\begin{df}
A convex continuous at $0$ function $\Phi:[0,\ii)\to [0,\ii)$ such that $\Phi(0)=0$ and $\Phi(u)>0$
 if $u\ne 0$  is called an {\it Orlicz function}.
\end{df}

\begin{rem}
(1) Since an Orlicz function is convex and continuous at $0$, it is necessarily continuous on $[0,\ii)$.

\noindent
(2) If $\Phi$ is an Orlicz function, then $\Phi(\lambda u) \leq \lambda \Phi(u)$ for all $\lambda \in [0, 1]$.
Therefore $\Phi$  is increasing, that is, $\Phi(u_1) < \Phi(u_2)$ whenever $0 \leq u_1 < u_2$.
\end{rem}

We will need the following lemma.

\begin{lm}\label{li} Let $\Phi$ be an Orlicz function. Then for any given $\dt>0$ there exists $t>0$
satisfying the condition
$$
t\cdot \Phi(u)\ge u \text{ \ whenever \ } u\ge \dt.
$$
In particular, $\lim_{u\rightarrow \infty}\Phi(u) = \infty$.

\end{lm}
\begin{proof}
Since $\Phi(u)>0$ as $u>0$, it is possible to find $a>0$ such that the equation $\Phi(u)=au$ has a solution $u=u_0>0$.
Then, as $\Phi$ is convex, we have $\Phi(u)\ge au$ for all $u\ge u_0$.

Fix $\dt>0$. If $\dt\ge u_0$, then we have
$$
\frac 1a \cdot \Phi(u)\ge u \ \ \ \forall \ u\ge \dt.
$$
 If $\dt<u_0$, then, since $\Phi(\dt)>0$ and
$\Phi$ is increasing on the interval $[\dt, u_0]$, there exists such $s>1$ that $s\cdot \Phi(u)\ge au$,
or
$$
\frac sa\cdot \Phi(u)\ge u, \ \ \forall \ u\ge \dt.
$$
\end{proof}

\begin{rem}
Since an Orlicz function $\Phi$ is continuous, increasing and such that $\lim_{u\rightarrow \infty}\Phi(u) = \infty$, there exists continuous increasing inverse function $\Phi^{-1}$ from $[0,\infty)$ onto $[0,\infty)$.
\end{rem}

If $\Phi$ is an Orlicz function, $x\in L^0_+$ and $x=\int_0^{\ii} \lb de_{\lb}$ its spectral decomposition,
one can define $\Phi(x)=\int_0^{\ii} \Phi(\lb) de_{\lb}$.
The {\it noncommutative Orlicz space} associated with $(\mc M,\tau)$ for an Orlicz function $\Phi$ is the set
$$
L^\Phi=L^\Phi(\mc M, \tau)=\left \{ x \in L^0(\mc M, \tau): \ \tau \left (\Phi\left (\frac {|x|}a \right )\right )
<\ii \text { \ for some \ } a>0 \right \}.
$$
The {\it Luxemburg norm} of an operator  $x \in L^{\Phi}$ is defined as
$$
\| x\|_\Phi=\inf \left \{ a>0:  \tau \left (\Phi\left (\frac {|x|}a \right )\right )  \leq 1 \right \}.
$$

\begin{teo}\label{t11} \cite[Proposition 2.5]{ku}.
 $(L^\Phi, \| \cdot \|_\Phi)$ is a Banach space.
\end{teo}

\begin{pro}\label{p12}
If $x \in L^\Phi$, then \ $\Phi(|x|) \in L^0$ and $\mu_t(\Phi(|x|))= \Phi(\mu_t(x))$, $t > 0$. In addition, $\tau (\Phi(|x|)) = \int_0^{\infty}\Phi(\mu_t(x)) dt$.
\end{pro}

\begin{proof}
As $x \in L^\Phi$, we have $\tau \left (\Phi\left (\frac {|x|}a \right )\right )<\ii$ for some $a>0$. This implies that
$\Phi\left (\frac {|x|}a \right ) \in L^1$, so
$
\tau\left (\left \{\Phi\left (\frac {|x|}a \right ) > \lambda\right \}\right ) < \ii
$
for all $\lambda >0$. Since
$$
\left \{\Phi\left (\frac {|x|}a \right ) > \lambda \right \} = \left \{\Phi^{-1}\left (\Phi\left (\frac {|x|}a \right )\right )
> \Phi^{-1}(\lambda)\right \} = \{|x| > a \Phi^{-1}(\lambda)\},
$$
it follows that  $\tau \left ( \{\Phi(|x|) > \mu \}\right ) = \tau\left (\{|x| > \Phi^{-1}(\mu) \}\right ) < \infty$
for all $\mu >0$, thus $\Phi(|x|) \in L^0$.

By \cite[Lemma 2.5, Corollary 2.8]{fk}, given $x \in L^0$,  we have $\mu_t(\varphi(|x|))= \varphi(\mu_t(x))$, $t > 0$, for every continuous increasing function $\varphi: [0, \infty) \rightarrow [0, \infty) $ with $\varphi(0)= 0$ and, in addition, $\tau (\varphi(|x)|) = \int_0^{\infty}\varphi(\mu_t(x)) dt$. Therefore $\mu_t(\Phi(|x|))= \Phi(\mu_t(x))$  and $\tau (\Phi(|x)|) = \int_0^{\infty}\Phi(\mu_t(x)) dt$.
\end{proof}

Next result follows immediately from Proposition \ref{p12}.

\begin{cor}\label{c11}
$L^\Phi= \{x \in L^0: \mu_t(x) \in L^\Phi(0,\ii)\}$ and $ \| x \|_\Phi =  \| \mu_t(x) \|_\Phi $ for all $x \in L^\Phi$.
\end{cor}

If $(L^\Phi(0,\ii), \| \cdot \|_\Phi)$  is the Orlicz function space on $(0,\ii)$ for an Orlicz function $\Phi$,
then, by \cite[Ch.2, Proposition 2.1.12]{es},
it is  a rearrangement invariant function space.
Since  $(L^\Phi(0,\ii), \| \cdot \|_\Phi)$ has the Fatou property \cite[Ch.2, Theorem 2.1.11]{es}, Corollary \ref{c11},  \cite[Theorem 4.1]{ddst},  and \cite[Theorem 3.4]{ddp} yield the following.

\begin{cor}\label{c12}
$(L^\Phi,  \| \cdot \|_\Phi)$  is a fully symmetric space with the Fatou property
and an  exact interpolation space for the Banach couple $(L^1,\mc M)$.
\end{cor}

We will also need the following property of the Luxemburg norm.
\begin{pro}\label{p12a}
If $x \in L_\Phi$ and $\| x\|_\Phi\leq 1$, then $\tau (\Phi(|x|) \leq \| x\|_\Phi$.
\end{pro}
\begin{proof}
By \cite[Ch.2, Proposition 2.1.10]{es}, $\int_0^{\infty}\Phi(|f|) dt\leq \| f\|_\Phi$ for $f\in L^\Phi(0,\ii)$
with $\| f\|_\Phi \leq 1$. Thus the result follows from Propsition \ref{p12} and Corollary \ref{c11}.
\end{proof}

\begin{df}
An Orlicz function $\Phi$ is said to satisfy {\it $\Delta_2-$condition} ({\it $\delta_2-$condition}) if there exist $k>0$ and $u_0\ge 0$ such that
$$
\Phi(2u)\leq k\Phi(u) \ \ \forall \ u\ge u_0 \ \ \ \text{(respectively,\ } \Phi(2u)\leq k\Phi(u) \ \ \forall \ u \in (0, u_0]).
$$

\end{df}
If an Orlicz function $\Phi$ satisfies $\Delta_2-$condition and $\delta_2-$condition simultaneously, we will say that $\Phi$
satisfies {\it $(\delta_2, \Delta_2)-$condition}. In this case $\Phi(2u)\leq c\Phi(u)$ for all $u \ge 0$ and some $c>0$.
Clearly, every space $L^p$, $1\leq p<\ii$, is the Orlicz  space for the function $\Phi(u)=\frac{u^p}p, \ u\ge 0$, which 
satisfies $(\delta_2, \Delta_2)-$condition.

\begin{rem}\label{r3}
(i) If an Orlicz function $\Phi$ satisfies $\Delta_2-$condition, then the Boyd index  $q_{L^\Phi(0,\ii)} < \infty $, that is,
it is non-trivial (see \cite[II, Ch.2, Proposition 2.b.5]{lt}).

\noindent
(ii) The function $\Phi_{\alpha}(u) = u \ ln^{\alpha}(e + u), \ \alpha \geq 0$, is an Orlicz function that
satisfies $(\delta_2, \Delta_2)-$condition for which the Boyd index $p_{L^\Phi(0,\ii)}$ is trivial,
that is,  $p_{L^\Phi(0,\ii)} =1$ \cite[\S 5]{ss}.

\end{rem}
A Banach space $(E, \| \cdot \|_E)  \su L^0$ is said to have {\it order continuous norm} if $\| x_{\al}\|_E\downarrow 0$
for every net $\{ x_{\al}\} \su E$ with $x_{\al}\downarrow 0$.

\begin{pro}\label{p14}
Let an Orlicz function $\Phi$ satisfy $(\delta_2, \Delta_2)-$condition. Then
\begin{enumerate}[(i)]
\item The fully symmetric space $(L^\Phi,  \| \cdot \|_\Phi)$ has order continuous norm.

\item The linear subspace $ L^1 \cap \mc M$  is dense in $(L^\Phi,  \| \cdot \|_\Phi)$.
\end{enumerate}
\end{pro}
\begin{proof}
(i) As shown in \cite[Ch.2, \S 2.1]{es},
the fully symmetric space $(L^\Phi(0,\ii), \| \cdot \|_\Phi)$  has order continuous norm.
Therefore, by \cite[Proposition 3.6]{ddp1}, the noncommutative fully symmetric space
$(L^\Phi,  \| \cdot \|_\Phi)$  also has order continuous norm.

(ii) Let $x \in L^\Phi_+$, $n = 1,2, \dots$, and $e_n$ the spectral projection corresponding to the interval
$(n^{-1},n)$. It is clear that $\{x e_n\}\su  \mc M$ and $e_n^\perp \downarrow 0$.
Also, by (i) and \cite[Theorem 3.1]{dps}, we have
$$
\|x - x e_n \|_\Phi=\| x e_n^\perp \|_\Phi \to 0 \text{ \ as \ } n\to \ii.
$$
Now, since  $\tau (\{x > \ep\}) < \infty$  for all $\varepsilon > 0$  (see proof of Proposition \ref{p12}),
it follows that $\{x e_n\} \su L^1$.

Since, for an arbitrary $ x \in L^\Phi$, we have $x = x_1 - x_2 + i(x_3 - x_4)$, where $x_i \in L^\Phi_+$, $i = 1,...,4$,
the assertion follows.
\end{proof}

\section{Main Results}

Let $\mc M$ be a semifinite von Neumann algebra with a faithful normal semifinite trace $\tau$,
$L^0=L^0(\mc M,\tau)$ the $*$-algebra of $\tau$-measurable operators affiliated with $\mc M$, $L^p=L^p(\mc M,\tau)$, $1\leq p\leq \ii$, the noncommutative $L^p-$space
associated with $(\mc M,\tau)$.

\begin{df}
Let $(X, \| \cdot\|)$ be a normed space,  and let $Y\su X$ be such that the neutral element of $X$ is an
accumulation point of $ Y$. A family of maps
$A_\al:X \to L^0$, $\al\in I$, is called {\it uniformly equicontinuous in measure (u.e.m) (bilaterally uniformly equicontinuous in measure (b.u.e.m)) at zero on $Y$} if for every $\ep>0$ and $\dt>0$ there is $\gamma>0$ such that,
given $x\in Y$ with $\| x\|<\gamma$, there exists $e\in \mc P(\mc M)$ such that
$$
\tau(e^{\perp})\leq \ep \text{ \ and \ }  \sup_{\al\in I}\| A_\al(x)e\|_{\ii}\leq \delta \ \ ( \text{respectively,\ }
\sup_{\al\in I}\| eA_\al(x)e\|_{\ii}\leq \delta).
$$
\end{df}
\begin{rem}
As explained in \cite[Introduction]{li}, in the commutative case, the notion of uniform equicontinuity  in measure at zero
of a family $\{ A_n\}_{n\in \Bbb N}$ coincides with the continuity in measure at zero
of the maximal operator associated with this family.
\end{rem}

\begin{df}
A sequence  $\{ x_n\}\su L^0$ is said to converge to $x\in L^0$ {\it almost uniformly (a.u.)} ({\it bilaterally almost
uniformly (b.a.u.)}) if for every $\ep>0$ there exists such a projection $e\in\PM$ that $\tau(e^{\perp})\leq \ep$ and
$\| (x-x_n)e\|_{\ii}\to 0$ (respectively, $\| e(x-x_n)e\|_{\ii}\to 0$).
\end{df}

A proof of the following fact can be found in \cite[Theorem 2.1]{li}.
\begin{pro}\label{p2}
Let $(X,\| \cdot \|)$ be a  Banach space, $A_n:X\to L^0$ a sequence of additive maps.
If the family $\{ A_n\}$ is u.e.m. (b.u.e.m.) at zero on $X$, then the set
$$
\{ x\in X: \{ A_n(x)\} \text{\ converges a.u. (respectively, b.a.u.)}\}
$$
is closed in $X$.
\end{pro}

\begin{df}
A linear map $T: L^1+L^{\ii}\to L^1+L^{\ii}$ such that
$$
\| T(x)\|_{\ii}\leq \| x\|_{\ii} \ \ \forall \ x\in \mc M \text{ \ and \ } \| T(x)\|_1\leq \| x\|_1 \ \ \forall \ x\in L^1.
$$
is called a {\it Dunford-Schwartz operator}.
\end{df}

If $T$ is a Dunford-Schwartz operator (positive Dunford-Schwartz operator), we will write $T\in DS$ (respectively, $T\in DS^+$).
If $T\in DS$, consider its ergodic averages
\begin{equation}\label{eq4}
A_n(x)=A_n(T,x)=\frac 1n \sum_{k=0}^{n-1} T^k(x), \ \ x\in  L^1+L^{\ii}.
\end{equation}

Here is a noncommutative  maximal ergodic inequality due to Yeadon \cite{ye}
(for the assumption $T\in DS^+$, see a clarification given in \cite[Proposition 1.1, Remark 1.2]{cl}):

\begin{teo}\label{t4}
Let $T\in DS^+$ and $A_n:L^1\to L^1$, $n=1,2,\dots$ be given by (\ref{eq4}). Then for every
$x\in L^1_+$ and $\nu>0$ there exists a projection $e\in \mc P(\mc M)$ such that
$$
\tau(e^{\perp})\leq \frac {\| x\|_1}{\nu} \text{ \ and \ } \sup_n\| eA_n(x)e\|_{\ii}\leq \nu.
$$
\end{teo}

Now, let $\Phi$ be an Orlicz function, $L^\Phi=L^\Phi(\mc M,\tau)$ the corresponding noncommutative
Orlicz space, $\| \cdot\|_\Phi$ the Luxemburg norm in $L_\Phi$.

As $L^\Phi$ is an exact interpolation space for the Banach couple
$(L^1,  \mc M)$ (see Corollary \ref{c12}),
\begin{equation}\label{eq5}
T(L^\Phi) \su L^\Phi  \text{ \ and \ } \|T\|_{ L^\Phi \to L^\Phi} \leq 1,
\end{equation}
hold for any $T\in DS$, and we have the following.

\begin{pro}\label{p4}
If $T\in DS^+$, then the family $A_n: L^\Phi \to L^\Phi, \ n=1,2,\dots$, given by (\ref{eq4})
is b.u.e.m. at zero on $(L^\Phi,\| \cdot \|_\Phi)$.
\end{pro}
\begin{proof}
It is easy to verify (see \cite[Lemma 4.1]{li}) that it is sufficient to show that $\{A_n\}$ is b.u.e.m.
at zero on  $(L^\Phi_+,\| \cdot \|_{\Phi})$.

Fix $\ep>0, \dt>0$. By Lemma \ref{li}, there exists $t>0$ such that
$$
t\cdot \Phi(\lb) \ge \lb \text{ \ as soon as \ } \lb \ge \frac {\dt}2.
$$
Let $\nu>0$ and $0<\gamma\leq 1$ be such that $\nu \leq \frac{\dt}{2t}$ and $\frac{\gamma}{\nu}\leq \ep$.

Take $x\in L^{\Phi}_+$ with $\| x\|_{\Phi}\leq \gamma$, and let $x=\int_0^{\ii}\lb de_{\lb}$ be its spectral decomposition.
Then we can write
$$
x=\int_0^{\dt/2}\lb de_{\lb}+\int_{\dt/2}^{\ii}\lb de_{\lb}\leq x_{\dt}+t\cdot \int_{\dt/2}^{\ii}\Phi(\lb) de_{\lb}
\leq x_{\dt}+t\cdot \Phi(x),
$$
where $x=\int_0^{\dt/2}\lb de_{\lb}$ and $\Phi(x)=\int_0^{\ii}\Phi(\lb) de_{\lb}$.

As $\| x_{\dt}\|_{\ii}\leq \frac{\dt}2$ and $T\in DS^+$, we have
$$
\sup_n\| A_n(x_{\dt})\|_{\ii}\leq \frac{\dt}2.
$$
Besides, by Proposition \ref{p12a}, $\| x\|_{\Phi}\leq 1$ implies that $\| \Phi(x)\|_1\leq \|x\|_M\leq \gamma$.
Since $\Phi(x)\in L^1_+$, in view of Theorem \ref{t4}, one can find a projection $e\in \mc P(\mc M)$ such that
$$
\tau(e^{\perp})\leq \frac{\| \Phi(x)\|_1}{\nu}\leq \frac{\gamma}{\nu}\leq \ep \text{ \ and \ }
\sup_n\| eA_n(\Phi(x))e\|_{\ii}\leq \nu \leq \frac{\dt}{2t}.
$$
Consequently,
$$
\sup_n\| eA_n(x)e\|_{\ii}\leq \sup_n\| eA_n(x_{\dt})e\|_{\ii}+t\cdot \sup_n\| eA_n(\Phi(x))e\|_{\ii}\leq \frac{\dt}2+
t\cdot  \frac{\dt}{2t}=\dt,
$$
and the proof is complete.
\end{proof}

Here is an individual ergodic theorem for noncommutative Orlicz spaces:

\begin{teo}\label{t5}
Assume that an Orlicz function $\Phi$ satisfy $(\dt_2,\Delta_2)-$condition. Then,
given $T\in DS^+$ and $x\in L^{\Phi}$, the averages (\ref{eq4}) converge b.a.u.  to some $\hat x\in L^{\Phi}$.
\end{teo}
\begin{proof}
Since, by Proposition \ref{p14}, the set $L^1\cap \mc M\su L^2$ is dense in $L^{\Phi}$
and the averages (\ref{eq4}) converge a.u., hence  b.a.u.,
for every $x\in L^2$ (see, for example, \cite[Theorem 4.1]{li}), it follows from Propositions \ref{p4} and \ref{p2} that
for any $x\in L^\Phi$ the averages (\ref{eq4}) converge b.a.u. to some $\hat x \in L^0$.

It is clear that a b.a.u. convergent sequence in $L^0$ converges in measure, hence $A_n(x)\to \hat x, \ x\in L^\Phi$, in measure.
Since, by Corollary \ref{c12}, $L^\Phi$ has the Fatou property, its unit ball is closed in the measure topology
\cite[Theorem 4.1]{ddst}, and (\ref{eq5}),
hence $\sup \limits_n \|A_n(x)\|_{ L^\Phi \to L^\Phi} \leq \|x\|_\Phi$, implies that $\hat x\in L^\Phi$.
\end{proof}

\begin{rem}
In was shown in \cite[Theorem 5.2]{cl} that if $E(0,\ii)$ is a fully symmetric function space with
Fatou property and non-trivial Boyd indices and $T\in DS^+$, then for any $x\in E(\mc M,\tau)$
the averages $A_n(x)$ converge b.a.u. to some  $\widehat{x} \in E(\mc M,\tau)$. According to Remark \ref{r3} (ii),
there exists   an Orlicz function $\Phi$ that satisfies $(\delta_2, \Delta_2)-$condition for which the Boyd index
$p_{L^\Phi(0,\ii)} $ is  trivial. Thus, Theorem \ref{t5} does not follow from Theorem \cite[Theorem 5.2]{cl}.

\end{rem}

Now we shall turn to a class of Orlicz spaces for which the averages (\ref{eq4}) converge a.u. The following
fundamental result is crucial.

\begin{teo}[Kadison's inequality \cite{ka}]
If $S:\mc M\to\mc M$ is a positive linear operator such that $S(\mathbf 1)\leq \mathbf 1$, then $S(x)^2\leq S(x^2)$ for
every $x^*=x\in \mc M$.
\end{teo}

\begin{df}
We call a convex function $\Phi$ on $[0,\ii)$ {\it $2-$convex} if the function $\wt \Phi(u)=\Phi(\sqrt u)$ is also convex.
\end{df}

For example, $\Phi(u)=\frac{u^p}p$, $u\ge 0$, is $2-$convex that satisfies $(\dt_2,\Delta_2)-$condition whenever $p\ge 2$.

It is clear that if $\Phi$ is a $2-$convex Orlicz  function, then $\wt \Phi$ is also an Orlicz function, and it is easy to
verify the following.

\begin{pro}\label{p5}
If $\Phi$ be a $2-$convex Orlicz function, then  $x^2\in L_{\wt \Phi}^+$
and $\| x^2\|_{\wt \Phi}=\| x \|_\Phi^2$ for every $x\in L^\Phi_+$.
\end{pro}

\begin{pro}\label{p6}
Let $\Phi$ be a $2-$convex Orlicz function. Then the family $\{A_n\}$ given by (\ref{eq4}) is u.e.m. at zero
on $(L^\Phi, \| \cdot \|_M)$.
\end{pro}

\begin{proof}
As it was noticed earlier, it is sufficient to show that $\{A_n\}$ is u.e.m. at zero on $(L^\Phi_+, \| \cdot \|_\Phi)$.

Fix $\ep>0$, $\dt>0$. By Proposition \ref{p4}, $\{ A_n\}$ is b.u.e.m. at zero on
$(L^{\wt \Phi}, \| \cdot \|_{\wt \Phi})$. Therefore there exists $\gm>0$  such that, given $y\in L^{\wt \Phi}$ with $\| y\|_{\wt \Phi}<\gm$,
$$
\sup_n\| eA_n(y)e\|_{\ii}\leq \dt^2 \text{ \ for some \ } e\in \mc P(\mc M) \text{ \ with \ } \tau(e^{\perp})\leq \ep.
$$

Now, let $x\in L^\Phi_+$ be such that $\| x\|_\Phi <\gamma^{1/2}$. Then, due to Proposition \ref{p5},
$x^2\in L^{\wt \Phi}$ and
$\| x^2\|_{\wt \Phi}=\| x\|_\Phi^2\leq \gamma$, implying that there is a projection $e\in \mc P(\mc M)$ such that
$$
\sup_n\| eA_n(x^2)e\|_{\ii}\leq \dt^2 \text{ \ and \ } \tau(e^{\perp})\leq \ep.
$$
Then, by Kadison's inequality,
$$
\left [ \sup_n\|A_n(x)e\|_{\ii} \right ]^2=\sup_n\|A_n(x)e\|_{\ii}^2=\sup_n\|eA_n(x)^2e\|_{\ii}\leq
$$
$$
\leq \sup_n\|eA_n(x^2)e\|_{\ii}\leq \dt^2,
$$
which completes the proof.
\end{proof}

Now, as in Theorem \ref{t5}, we obtain the following.

\begin{teo}\label{t6}
If an Orlicz function $\Phi$ satisfies $(\dt_2,\Delta_2)-$condition and is $2-$convex, then, given $T\in DS^+$ and $x\in L^{\Phi}$, the averages (\ref{eq4}) converge a.u. to some $\hat x\in L^{\Phi}$.
\end{teo}

\end{document}